\documentclass[12pt,leqno,fleqn]{amsart}  
\usepackage{amsmath,amstext,amsthm,amssymb,amsxtra,comment}

\usepackage[top=1.5in, bottom=1.5in, left=1.25in, right=1.25in]	{geometry}
\usepackage{txfonts} 
\usepackage[T1]{fontenc}
\usepackage{lmodern}

 \usepackage{euler}   

\usepackage{tikz}
\usepackage{paralist}   

\usepackage{mathtools}
\mathtoolsset{showonlyrefs,showmanualtags}

\makeatletter
\DeclareRobustCommand\widecheck[1]{{\mathpalette\@widecheck{#1}}}
\def\@widecheck#1#2{%
    \setbox\z@\hbox{\m@th$#1#2$}%
    \setbox\tw@\hbox{\m@th$#1%
       \widehat{%
          \vrule\@width\z@\@height\ht\z@
          \vrule\@height\z@\@width\wd\z@}$}%
    \dp\tw@-\ht\z@
    \@tempdima\ht\z@ \advance\@tempdima2\ht\tw@ \divide\@tempdima\thr@@
    \setbox\tw@\hbox{%
       \raise\@tempdima\hbox{\scalebox{1}[-1]{\lower\@tempdima\box
\tw@}}}%
    {\ooalign{\box\tw@ \cr \box\z@}}}
\makeatother

\usepackage{hyperref} 
\hypersetup{
    colorlinks=true,       
    linkcolor=blue,          
    citecolor=magenta,        
    filecolor=magenta,      
    urlcolor=cyan           
}

\usepackage{amsrefs}

\theoremstyle{plain} 
\newtheorem{lemma}[equation]{Lemma} 
\newtheorem{proposition}[equation]{Proposition} 
\newtheorem{theorem}[equation]{Theorem} 
\newtheorem{corollary}[equation]{Corollary} 

\newtheorem{priorResults}{Theorem}

\theoremstyle{definition}

\theoremstyle{remark}

\numberwithin{equation}{section}

\title[$ \ell ^{p}$ improving for Spherical Averages ] {
$ \ell ^{p}$-improving inequalities  for Discrete Spherical Averages } 

\author[R. Kesler]{Robert Kesler}   

\address{ School of Mathematics, Georgia Institute of Technology, Atlanta GA 30332, USA}
\email {rkesler6@mail.gatech.edu}

\author[M.~T. Lacey]{Michael T.  Lacey}   

\address{ School of Mathematics, Georgia Institute of Technology, Atlanta GA 30332, USA}
\email {lacey@math.gatech.edu}
\thanks{Research supported in part by grant  from the US National Science Foundation, DMS-1600693 and the 
Australian Research Council ARC DP160100153.}

\begin{document}

\begin{abstract}
 Let $ \lambda ^2 \in \mathbb N $, and in dimensions $ d\geq 5$, 
 let  $ A _{\lambda } f (x)$ denote the average of $ f \;:\; \mathbb Z ^{d} \to \mathbb R $ over 
 the lattice points on the sphere of radius $\lambda$ centered at $x$.    
 We prove   $ \ell ^{p}$ improving properties of 
 $ A _{\lambda }$. 
 \begin{equation*}
\lVert A _{\lambda }\rVert _{\ell ^{p} \to \ell ^{p'}} \leq C _{d,p, \omega (\lambda ^2 )} \lambda ^{d ( 1-\frac{2}p)}, 
\qquad    \tfrac{d-1}{d+1} < p \leq \frac{d} {d-2}.  
\end{equation*}
It holds  in dimension $ d =4$ for odd $ \lambda ^2 $.  The dependence is in terms of $ \omega (\lambda ^2 )$, the number of distinct prime 
factors of $ \lambda ^2 $.  
These inequalities are  discrete versions of a classical inequality of Littman and Strichartz on the $ L ^{p}$ improving property of spherical averages on $ \mathbb R ^{d}$.  In particular they are scale free, in a natural sense.  
The proof uses the decomposition of the corresponding multiplier  whose properties were established by Magyar-Stein-Wainger, and Magyar. 
We then use a proof strategy of Bourgain, which dominates each part of the decomposition by an endpoint estimate. 

\end{abstract}
	\maketitle

\section{Introduction} 

The subject of this paper is in discrete harmonic analysis, in which continuous objects are studied in the setting of the integer lattice. Relevant norm properties are much more intricate, with novel 
distinctions with the continuous case arising.  

In the continuous setting, $ L ^{p}$-improving properties of averages over lower dimensional surfaces are widely recognized as an essential property of such averages \cites{MR0358443,MR1969206,MR1654767,MR0256219}. It  continues to be very active subject of investigation.  
In the discrete setting, these questions are largely undeveloped. 
They are implicit in work on discrete fractional integrals  by several authors  \cites{MR2872554,MR1945293,MR1825254,MR1771530}, as well as two recent papers \cites{MR3933540,MR3892403} on sparse bounds for  discrete singular integrals.

Our  main results concern $\ell^p$ improving estimates for  averages over discrete spheres, in dimensions $ d \geq 5$, and in dimension $ d=4$, for certain radii.    

We recall the continuous case. For dimensions $d\geq 2$, let $ d\sigma $ denote Haar measure on the sphere of radius one, and set $  \mathcal   A_1 f = \sigma \ast f $ be convolution with respect to $ \sigma $.  
The classical result of Littman \cite{MR0358443} and Strichartz \cite{MR0256219} gives the sharp $ L ^{p}$ improving property of this average.   
Here, we are stating the result in a  restrictive way, but the full strength is obtained by interpolating with the obvious $ L ^{1} \to L ^{1}$ bound.   

\begin{priorResults}\label{t:LS} \cites{MR0358443,MR0256219}  For dimensions $ d \geq 2$,  we have 
$ \lVert  \mathcal A_1 \rVert _{ \frac{d+1}{d}  \to d+1}$.  
\end{priorResults}

We study the discrete analog of $ \mathcal A_1 f $ in higher dimensions.  
For  $ \lambda ^2  \in \mathbb N $, 
let $\mathbb S^d_\lambda := \{ n\in \mathbb{Z}^d \;:\; \lvert n\rvert = \lambda\}$. 
For   a   function $ f$ on $ \mathbb Z ^{d}$, define 
\begin{equation*}
A _{\lambda } f (x) = 
\lvert \mathbb S^d_\lambda \rvert ^{-1}
\sum_{n \in  \mathbb S^d_\lambda } f    (x-n). 
\end{equation*}
The study of the harmonic analytic properties of these averages was initiated by Magyar 
\cite{MR1617657}, with Magyar, Stein and Wainger \cite{MR1888798} proving a discrete variant of the Stein spherical maximal function theorem \cite{MR0420116}.  
This result holds in dimensions $ d \geq 5$, as irregularities in the number of lattice points on spheres presenting obstructions to a positive result in dimensions $ d =2,3,4$.  
In particular, they proved the result below. See Ionescu  \cite{MR2053347} for an endpoint result, and the work of several others which further explore this topic \cites{MR1925339,MR2346547,160904313,MR3819049,MR3960006}.

\begin{priorResults}\label{t:MSW}[Magyar, Stein, Wainger, \cite{MR1888798}]
For $ d \geq 5$, there holds 
\begin{equation*}
\bigl\lVert  \sup _{\lambda }  \lvert    A _{\lambda } f \rvert \bigr\rVert _{p} \lesssim \lVert f\rVert _{p}, \qquad   \tfrac{d} {d-2} < p < \infty . 
\end{equation*}
\end{priorResults}

\noindent We will refer to $ p  _{\textup{MSW}}= \frac{d} {d-2}$ as the Magyar Stein Wainger index.

Our first main result is a discrete variant of the result of Littman and Strichartz above.  
First note that   $ A _{\lambda }$ is clearly bounded from $ \ell ^{p}$ to $ \ell ^{p}$, 
for all $ 1\leq p \leq \infty $.  Hence, it trivially improves any $ f \in \ell ^{p} (\mathbb Z ^{d})$ to an $ \ell ^{\infty }(\mathbb Z ^{d})$ function.   
But, proving a  \emph{scale-free version} of the inequality is not at all straightforward. 

In dimensions $ d=4$, there is an arithmetical obstruction, namely for certain radii $ \lambda $, the number of points on the sphere of radius $ \lambda $ can be very small.  To address this, let 
$\Lambda_d := \{ 0 < \lambda < \infty \;:\; 
\lambda^2\in \mathbb{N}\}$, for $d\geq 5$, 
and for $d=4$, 
\begin{equation}\label{e:Lam}
\Lambda_4 := 
\{ 0 < \lambda < \infty \;:\; 
\lambda^2\in \mathbb{N}\setminus 4\mathbb{N}\}
\end{equation}
Following the work of Magyar \cite{MR1925339}, we will address the case of dimension $ d=4$ below. 
And, we will prove results \emph{below the Magyar Stein Wainger index.}

\begin{theorem}\label{t:improve}  
  In dimensions $ d \geq 4$,    the inequality below holds for all $ \lambda \in \Lambda _d$.  
\begin{equation} \label{e:improve}
\lVert A _{\lambda }\rVert _{p \to p'} \leq C_{d,p, \omega (\lambda ^2 )} \lambda ^{d \left( 1 - \frac2{p} \right)}, 
\qquad  \tfrac{ d+1}  {d - 1} < p  \leq 2.  
\end{equation}
Above, $ \omega (\lambda ^2 )$ is the number of distinct prime factors of $ \lambda ^2 $. 
In order that \eqref{e:improve} hold, it is necessary that $ p \geq  \frac{d +1}{d}$, for $ d\geq 5$.  

\end{theorem}

This Theorem was independently proved by Hughes \cite{180409260H}. 
The proof herein  uses the same elements, but optimizes the interpolation part of the argument.  It is short, and simple enough that one can give concrete estimates for the dependence on $ \lambda $, which we indicate below.

To explain our use of the phrase `scale free'   we make this definition.
For a cube $ Q \subset \mathbb R ^{d}$ of volume at least one, we set localized and normalized norms to be 
\begin{equation} \label{e:localNorm}
\langle f \rangle _{Q, p} := \Bigl[
\lvert  Q\rvert ^{-1} 
\sum_{n \in Q \cap \mathbb Z ^{d}} \lvert  f (n)\rvert ^{p} 
\Bigr] ^{1/p}, \qquad 0< p \leq \infty . 
\end{equation} 
An equivalent way to phrase our theorem above is the following corollary. Note that in this language, the 
inequalities in \eqref{e:Improve} are uniform in the choice of $ \lambda $.  

\begin{corollary}\label{c:fixed} Let $ d\geq 4$, and set $ \mathbf I_d$ to be the open triangle with vertices 
$ (0,1)$, $ (1,0)$, and $ (\frac{d-1} {d+1}, \frac{d-1} {d+1})$.  (See Figure~\ref{f:IS}.)
For $ (1/p_1, 1/p_2) \in \mathbf I_d$, there is a finite constant $ C = C _{d, p_1 ,p_2, \omega (\lambda ^2 )}$ so that  
\begin{equation}\label{e:Improve}
\langle A _{\lambda } f _1, f_2 \rangle   \leq C   \langle f_1 \rangle _{Q,p_1} \langle f_2\rangle _{Q,p_2}
\lvert  Q\rvert, \qquad    \lambda \in \Lambda_d.  
\end{equation}

\end{corollary}

\begin{figure}
\begin{tikzpicture}[scale=1.5]
\draw [->] (0,-.5) -- (0,3.5) node[left] {\textsubscript{$ {1}/ {p_2}$}}; 
\draw [->] (-.5,0) -- (3.9,0) node[below] {\textsubscript{$ {1}/ {p_1}$}}; 
\draw (3,.05) -- (3,-.25) node[right] {\textsubscript{1}};
\draw (.05,3) -- (-.25,3) node[below] {\textsubscript{1}};
\draw[dotted]   (0,3) -- (2.5, 2.5) node[above,right] {\textsubscript{$ (\frac{d-1} {d+1},\frac{d-1} {d+1} )$}}-- (3,0) -- (0,3);  

\draw   (1.6,1.6) node[right] {$ \mathbf I_d$  };  
\filldraw (2,2) circle (.05em) node (MSW) {};  
\draw (5, 1.5 ) node { $ \bigl( \frac 1 {p  _{\textup{MSW}}},\frac 1 {p  _{\textup{MSW}}}\bigr)$, see \eqref{e:KLM}};
\draw[->]  (3.7,1.5) to [out = 180 , in = -30] (MSW) ; 
\end{tikzpicture}

\caption{ 
The triangle $ \mathbf I_d$ of Theorem~\ref{c:fixed}, for the  $ \ell ^{p}$ improving inequality \eqref{e:improve}, is the dotted triangle with corners $ (0,1)$ to $ P_1$ to $ (1,0)$.   The diagram above is for the case of dimension $ d \geq 5$.  The point closest to the diagonal corresponds to the 
Magyar Stein Wainger index. At this point the maximal inequality \eqref{e:KLM} holds. 
}
\label{f:IS}
\end{figure}
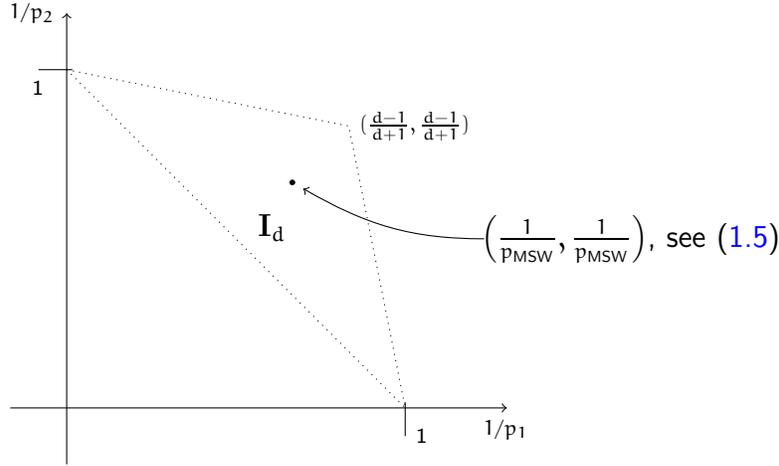

Our main inequality is only of interest for $ \tfrac{d+1} {d-1} < p < \frac {d} {d-2} = p _{\textup{MSW}}$, in the case of $ d\geq 5$.  
Indeed, at $ p _{\textup{MSW}}$, we know a substantially better result.   
For indicators functions $ f = \mathbf 1_{F}$ and $ g = \mathbf 1_{G}$ supported in a cube $ E$ of side length $ \lambda _0$, 
we have \cite{181002240}  this restricted maximal estimate at the index $ p _{\textup{MSW}}$. 
\begin{equation}  \label{e:KLM}
 \bigl\langle  \sup _{\lambda _0/2  < \lambda < \lambda _0 } 
 A _{\lambda } f , g  \bigr\rangle \lesssim \langle f \rangle _{E, \frac{d} {d-2}} \langle g \rangle _{E,  \frac{d} {d-2}} \lvert  E\rvert .  
\end{equation}

The proof of \eqref{e:improve} requires a  circle method decomposition of $ A _{\lambda }$ in terms of its Fourier multiplier.  The key elements here were developed by Magyar, Stein and Wainger \cite{MR1888798}, with  additional observations of Magyar \cite{MR2287111}. 
We recall this in  \S\ref{s:decompose}. 
The short proof in \S\ref{s:proof}
uses indicator functions, following work of Bourgain \cite{MR812567}, and in the discrete setting Ionescu \cite{I}, and 
Hughes \cite{MR3671577}.  
We comment briefly on sharpness in the last section of the paper.

\smallskip 
We acknowledge useful conversations with Alex Iosevich and Francesco Di Plinio on the topics of this paper.  
Fan Yang and the referee suggested several improvements of the paper.

\section{Decomposition} \label{s:decompose}

Throughout $ e (x)= e ^{2 \pi i x}$.    
The Fourier transform on $ \mathbb Z ^{d}$ is given by 
\begin{equation} \label{e:FT}
 \widehat f  (\xi )  = \sum_{x\in \mathbb Z ^{d}} e (-\xi \cdot x) f (x), \qquad \xi \in \mathbb T ^{d} \equiv [0,1] ^{d}. 
\end{equation}
We will write  $ \widecheck \phi $ for the inverse Fourier transform. 
The Fourier transform on $ \mathbb R ^{d}$ is 
\begin{equation} \label{e:FR}
\widetilde   \phi  (\xi ) = \int _{\mathbb R ^{d}} e (- \xi \cdot  x  ) f (x) \; dx . 
\end{equation} 

We work exclusively with convolution operators $ K :  f \mapsto  \int _{\mathbb T ^{d}} k (\xi  ) \widehat f (\xi  ) e ( \xi \cdot  x  )\;d \xi $.  In this notation, $ k $ is the multiplier, and the convolution is 
$ \widecheck k \ast f $.  
Lower case letters are frequently, but not exclusively, used for the multipliers, and capital letters for the  corresponding convolution operators.

The following estimate for the number of lattice points on a sphere holds. 
\begin{equation}\label{e:simeq}
\lvert \mathbb{S}_n^d\rvert = \lvert  \{ n \in \mathbb Z ^{d } \;:\; \lvert  n\rvert = \lambda  \}\rvert  \simeq \lambda ^{d-2}, 
\qquad \lambda \in \Lambda_d.    
\end{equation}
Redefine  the discrete spherical averages $ A _{ \lambda } f $ to be 
\begin{align}   \label{e:normalA}
A _{\lambda } f (x) &= \lambda ^{-d+2 } 
\sum_{n \in \mathbb Z ^{d } \;:\; \lvert n\rvert  = \lambda  } f (x-n)
\\
&= \int _{\mathbb T ^{d}} a _{\lambda } (\xi ) \widehat f (\xi ) e (\xi \cdot  x  )\; d \xi 
\\
\textup{where} \quad 
a _{\lambda } (\xi ) & = 
\lambda ^{-d+2}
\sum_{n \in \mathbb Z ^{d } \;:\; \lvert n\rvert = \lambda  } 
e (\xi \cdot n).
\end{align}
The  decomposition of $ a _{ \lambda }  $ into a `main' term $ c _{\lambda } $ and an `residual' term $ r _{\lambda }   = a _{\lambda }   - c _{\lambda }$  follows  development of Magyar, Stein and Wainger \cite{MR1888798}*{\S5}, Magyar \cite{MR2287111}*{\S4} and Hughes \cite{160904313}*{\S4}.  We will be very brief. 

For integers $ q$, set  $ \mathbb Z_q ^{d} = (\mathbb Z  / q \mathbb Z )^{d}$. Set  $  \mathbb Z _q ^{\times } = \{a \in \mathbb Z _q  \;:\; (a,q)=1\}$ to be the multiplicative group. 
We have 
\begin{align}\label{e:a1}
c _{\lambda } (\xi )  &= \sum_{1\leq q \leq \lambda }   c _{\lambda ,q} (\xi ) , 
\\    \label{e:Cdef}
c _{\lambda , q } (\xi )&= \sum_{\ell \in \mathbb Z_q^d} 
K (\lambda ,  q,  \ell )  \Phi _{q} (\xi - \ell /q)  \widetilde  {d \sigma _{\lambda }} ( \xi - \ell /q), 
\\
\label{e:Kdef}
K (\lambda ,  q,  \ell ) & =    
 q ^{-d} \sum_{ a \in \mathbb Z^\times_q } \sum_{n\in \mathbb Z_q^d} e_q  \bigl(-a\lambda^2  + \lvert  n\rvert ^2 a + n \cdot \ell  \bigr)  . 
\end{align}
Above, $ \Phi $ is a smooth non-negative  radial bump function, $ \mathbf 1_{[-1/8,1/8] ^{d}} \leq \Phi \leq \mathbf 1_{[-1/4,1/4] ^{d}}$.  Further, $ \Phi _{q} (\xi ) = \Phi (q \xi )$.  
Throughout we use $ e _{q} (x) =  e (x/q) = e ^{2 \pi i x/q}$.  
The term in \eqref{e:Kdef} is a Kloosterman sum, a fact that is hidden in the expression above, but becomes clear after exact summation of the quadratic Gauss sums.  
In addition, $ d \sigma _{\lambda }$ is the continuous unit Haar measure on the sphere of radius $ \lambda $ in $ \mathbb R ^{d}$.  
Recall  the stationary phase estimate 
\begin{equation}\label{e:stationary}
\lvert   \widetilde {d \sigma _{\lambda }} (\xi ) \rvert \lesssim  \lvert  \lambda \xi \rvert ^{- \frac{d-1}2}.   
\end{equation}

Essential here is the \emph{Kloosterman refinement}. 
The estimate below goes back to the work of Kloosterman \cite{MR1555249} and Weil \cite{MR0027006}.
Magyar \cite{MR2287111}*{\S4} used it in this kind of setting.  (It is essential to the proof of Lemma~\ref{l:Akos}.)

\begin{lemma}\label{l:K}  \cite{MR2287111}*{Proposition 7} For all  $ \eta >0$, and all $ 1 \leq q \leq \lambda $,  $ \lambda \in \Lambda_d$,  
\begin{equation}\label{e:K<}
\sup _{\ell } \lvert  K (\lambda ,  q,  \ell ) \rvert \lesssim _{\eta } q ^{- \frac{d-3}2 + \eta } 
 \rho (q, \lambda ) , 
\end{equation}
 where we write $ q = q_1 2 ^{r}$, with $ q_1$ odd, so that $ \rho (q, \lambda ) = \sqrt { (q_1, \lambda ^2  ) 2 ^{r}}$, where $ (q_1, \lambda ^2 )$ is the greatest common divisor of $ q_1 $ and $ \lambda ^2 $.  The implied constant only depends upon $ \eta >0$. 
\end{lemma}

Concerning the terms $ \rho (q, \lambda ) $, we need this Proposition. 

\begin{proposition}\label{p:rho}  We have for $ N<   \lambda  $ and $ a>1$, and all $ \eta  >0$
\begin{align}  \label{e:rho}
\sum_{ q \;:\; N \leq q   } q ^{-a} \rho (\lambda ,q)   &\lesssim   
 N ^{1-a}  \sigma _{-1/2} (\lambda ^2 )
, 
\\ \label{e:rho2}
\sum_{ 1\leq q \leq N}  q ^{\eta } \rho (\lambda ,q) &\lesssim 
 N ^{1+ \eta }   \sigma _{ - 1/2 } (\lambda ^2 ). 
\end{align}
Above $ \sigma _{b} (n ) = \sum_{d \;:\; d\,\vert\,n} d ^{b}$ is the generalized sum of divisors function. 
\end{proposition}

\begin{proof}

Write $ q = 2 ^{r}s t$, where  $ s $ and $ t$ are odd, $ r \geq 0$ and 
$ ( s , \lambda ^2 )=1$.  With this notation, $ \rho (\lambda ,q) = t 2 ^{r}$.   For \eqref{e:rho}, the sum we need to estimate is 
\begin{align*}
\sum_{ t\;:\; t\,|\, \lambda ^2 } \sum_{s=1} ^{\infty } \sum_{\substack{r=0\\  2 ^{r} st\geq N }} ^{\infty } 
\frac {[ t 2 ^{r} ] ^{\frac{1 }2} } { [t s 2 ^{r}] ^{a}} .  
\end{align*}
We will sum over $ r$ first.  There is first the cases in which $ st \leq N$: 
\begin{align*}
\sum_{t \;:\; t\,|\, \lambda ^2 } \sum_{s=1} ^{\infty } \sum_{\substack{r=0\\  2 ^{r} st\geq N , \  st \leq N}} ^{\infty } 
\frac {[ t 2 ^{r} ] ^{\frac{1 }2} } { [ st 2 ^{r}] ^{a}} 
& \lesssim 
\sum_{ t\;:\; t\,|\, \lambda ^2 } \sum_{ \substack{s=1\\  st \leq N} } ^{\infty } \Bigl(\frac{st}N \Bigr) ^{a - 1/2} \frac{1} {s ^{a} t ^{a - 1/2}} 
\\
& \lesssim N ^{ 1/2 -a }\sum_{t \;:\; t\,|\, \lambda ^2 } \bigl(N/t \bigr) ^{1/2}  
= N ^{1-a} \sum_{t \;:\; t\,|\, \lambda ^2 } t ^{-1/2}  
\\
& \lesssim N ^{1-a} \sigma _{-1/2} (\lambda ^2 ).   
\end{align*}

The second case of $ st > N$ imposes no restriction on $ r$. The sum over $ r \geq 0$ is just a geometric series,  therefore we have to bound 
\begin{align}
\sum_{t \;:\; t\,|\, \lambda ^2 } \sum_{ \substack{s=1\\  st >  N} } ^{\infty }  \frac{1} {s ^{a} t ^{a- 1/2}} 
& \lesssim \sum_{t \;:\; t\,|\, \lambda ^2 }  \Bigl( \frac{t}N  \Bigr) ^{a-1}  \frac{1} {  t ^{a - 1/2}} 
\\  \label{e:loglog}
& \lesssim N ^{1-a} \sum_{t \;:\; t\,|\, \lambda ^2 }  \frac1{\sqrt t } \lesssim N ^{1-a }  \sigma  _{-1/2} (\lambda ^2 ).
\end{align}

\smallskip
We turn to \eqref{e:rho2} using the notation above.  We estimate 
\begin{align*}
\sum_{t \;:\; t\,|\, \lambda ^2 } \sum_{s=1} ^{\infty } \sum_{\substack{r=0\\   2 ^{r} st\leq N }} ^ \infty 
 [  2 ^{r}st] ^{\eta } [ 2 ^{r}t] ^{ \frac{1}2} 
 & \lesssim 
 \sum_{t \;:\; t\,|\, \lambda ^2 } \sum_{ \substack{ s=1\\   st \leq N}} ^{\infty }    [st] ^{\eta } t ^{\frac{1}2}  (N/st) ^{ \frac{1}2+ \eta } 
 \\
 & \lesssim N ^{\frac{1}2 + \eta }  \sum_{t \;:\; t\,|\, \lambda ^2 } \sum_{ \substack{ 1\leq s \leq N/t }}     s ^{  - \frac{1}2} 
 \\
 & \lesssim  N ^{1+ \eta }  \sum_{t \;:\; t\,|\, \lambda ^2} t ^{ - \frac{1}2 } \lesssim N ^{1+  \eta } \sigma  _{-1/2} (\lambda ^2 ). 
\end{align*}

\end{proof}

The `main' term is $ C _{\lambda } f$, and the `residual' term is 
$
R _{\lambda } = A _{\lambda }  - C _{\lambda } 
$. 
This is a foundational estimate for us.  (The reader should note that the normalizations 
    here and in \cite{MR2287111} are different.)

\begin{lemma}\label{l:Akos}\cite{MR2287111}*{ Lemma 1, page 71}
We have, for all $ \epsilon >0$, uniformly in $ \lambda  \in \Lambda_d $,  
$ \lVert R _{\lambda }\rVert_ {2\to 2} \lesssim _{\epsilon}  \lambda ^{\frac{1-d}2+ \epsilon }$. 

\end{lemma}

For a multiplier $ m$ on $ \mathbb T ^{d}$, define a family of related multipliers by 
\begin{equation}\label{e:mq}
m _{\lambda , q } = \sum_{\ell \in \mathbb Z^d_q} 
K (\lambda ,  q,  \ell )  m (\xi - \ell /q).  
\end{equation}
We estimate the Fourier transform here.  

\begin{proposition}\label{p:mq} For a multiplier $ m_{\lambda , q}$ as in \eqref{e:mq}, 
we have 
\begin{equation} \label{e:Mq}
 \lvert   \widecheck m _{\lambda ,q} (n) \rvert
\leq q \lvert    \widecheck m (n)  \rvert . 
\end{equation} 
\end{proposition}

We include a proof for convenience. 
\begin{proof}
Our needs here are no different than those of \cites{MR1888798,I}. See for instance the argument after \cite{I}*{(2.9)}.   
Rewrite the Kloosterman sum in \eqref{e:Kdef} in terms of Gauss sums, namely 
\begin{align}   \label{e:Gauss1}
K (\lambda ,  q,  \ell ) & = \sum_{ a \in \mathbb Z^\times_q }  e_q (-a\lambda^2  )G (a/q, \ell ) , 
\\   \label{e:Gauss2}
\textup{where} \quad G (a/q, \ell ) & :=  q ^{-d}   \sum_{n\in \mathbb Z_q^d} e_q \bigl( \lvert  n\rvert ^2 a + n \cdot \ell  \bigr)  . 
\end{align}
Observe that $ G (a/, \cdot )$  is a Fourier transform on the group $ \mathbb Z^d_q$.  Namely, if $ \phi  ( \ell ) = e ( \lvert  \ell \rvert ^2 a/q)$ is the function on 
$ \mathbb Z^d_q$, we have $ \widehat \phi (- \ell ) = \widehat \phi (\ell ) = G (a/q, \ell ) $. 
Using the version formula on that group we have 
\begin{equation}
\label{e:G2}
\sum_{\ell \in \mathbb Z^d_q } G (a/q, \ell ) e_q ( y \cdot \ell )  = e_q (\lvert  y\rvert ^2 a), \qquad y \in \mathbb Z ^{d}_q .
\end{equation}

Define 
\begin{equation} \label{e:m}
m ^{a/q} (\xi ) =  e _q (-\lambda ^2 a) \sum_{ \ell  \in \mathbb Z^d_q} G (a/q, \ell ) m (\xi - \ell /q), \qquad 
 a \in \mathbb Z _q ^{\times }. 
\end{equation}
By \eqref{e:G2}, we have 
\begin{align*}
\widecheck {m ^{a/q}} (n) 
& = \int _{\mathbb T ^{d}}  m ^{a/q} (\xi ) e (- \xi \cdot n) \; d \xi 
\\
 & =  e_q (-\lambda ^2 a) 
  \int _{\mathbb T ^{d}} 
 \sum_{ \ell  \in \mathbb Z^d_q} G (a/q, \ell ) m (\xi - \ell /q) e (- \xi \cdot n) \; d \xi 
 \\
 &= e_q (  (\lvert  n \rvert ^2  -\lambda ^2 ) a) \widecheck m (n) 
\end{align*}
Take the absolute value, and sum over $ q\in \mathbb Z ^{\times }_q$ to conclude the Proposition. 
\end{proof}

\section{Proof}\label{s:proof}

It suffices to show this. For $ f = \mathbf 1_{F} \subset  E = [0, \lambda ] ^{d}\cap \mathbb Z ^{d}$, 
choices of $ 0< \epsilon <1$,    and integers $ N$ we can write 
\begin{gather}\label{e:Absorb}
A _{\lambda } f \leq M_1 + M_2 , 
\\ \label{e:xM1}
\textup{where} \quad  \langle M_1 \rangle _{E, \infty } \lesssim N  ^{2} \langle f \rangle _{E} 
\\ \label{e:xM2}
\textup{and} \quad  \langle M_2 \rangle _{E,2} \lesssim _{\epsilon } 
  N ^{ \epsilon + \frac{3-d}2} \sigma_  {- 1/2 }  (\lambda ^2 ) \cdot \langle f \rangle_E ^{1/2} .
\end{gather}
Above, $ \sigma _ {- 1/2 } (\lambda ^2 )$ is the generalized sum of divisors function, in Proposition~\ref{p:rho}. 

A straight forward argument concludes the proof from here, by optimizing over $ N$. 
Indeed, for $ g = \mathbf 1_{G}$ with $ G \subset E$, we have for any integer $ N$, 
\begin{equation*}
\lvert  E\rvert ^{-1}  \langle A _{\lambda } f, g \rangle \lesssim _{\epsilon } N  ^2 \langle f \rangle_E \langle g \rangle_E +
N ^{\epsilon +\frac{3-d}2  } \sigma_ {- 1/2 } (\lambda ^2 ) \bigl[ \langle f \rangle _{E}  \langle g \rangle_E \bigr] ^{1/2} . 
\end{equation*}
Minimizing over $ N$, we see that we should take  
\begin{equation*}
N ^{\frac{d+1}2 - \epsilon } \simeq  \sigma _{- 1/2 } (\lambda ^2 ) \bigl[ \langle f \rangle _{E}  \langle g \rangle_E \bigr] ^{ - \frac{1}2}. 
\end{equation*}
With this choice of $ N$, we see that 
\begin{equation*}
\lvert  E\rvert ^{-1}  \langle A _{\lambda } f, g \rangle \lesssim _{\epsilon } 
 \sigma _ { - 1/2 }(\lambda ^2 ) ^{\frac{4} {d+1} + \epsilon '} 
  \bigl[ \langle f \rangle _{E}  \langle g \rangle_E \bigr] ^{ \frac{d-1} {d+1} + \epsilon '}.
\end{equation*}
Above, $ \epsilon ' = \epsilon ' (\epsilon )$ tends to zero as $ \epsilon $ does.  
This is a restricted weak type inequality.  Interpolation with the obvious $ \ell ^2 $ bound completes the proof of our Theorem. 
We remark that this gives a concrete estimate of the dependence on $ \lambda ^2 $.  
We have 
\begin{equation*}
\sigma  _{- 1/2 } (n) \leq \prod _{j=1} ^{\omega (n )} (1 - \tfrac1{ \sqrt{p_j}  }) ^{-1}  \lesssim e ^{c \frac{\sqrt {\omega (n)}} {\log \omega (n)}}, 
\end{equation*}
where $ 2= p_1 < p_2 < \cdots $ is the increasing ordering of the primes.  This is at most a constant depending upon $ \omega (n)$, 
the number of distinct prime factors of $ n$.

\smallskip 
We turn to the construction of $ M_1$ and $ M_2$. 
If $ \lambda \leq N$, we set $ M_1 = A _{\lambda } f$.  Since we normalize the spherical averages by $ \lambda ^{d-2}$, \eqref{e:xM1} is immediate.      

Proceed under the assumption that $ N < \lambda $, and write $ A _{\lambda } = C _{\lambda } + R_{\lambda }$, with $ c _{\lambda }$ defined in \eqref{e:a1}.  The first contribution to $ M_2$ is $ M _{2,1} = R _{\lambda } f$. 
By Lemma~\ref{l:Akos},  this satisfies \eqref{e:xM2}.  (We do not need the arithmetic function $ \sigma _ {-1/2} (\lambda ^2 )$ in this case.)  
Turn to $ C _{\lambda }$. The second contribution to $ M_2$ is the `large $ q$' term  
\begin{equation} \label{e:M22}
M _{2,2} = \sum_{N\leq q \leq \lambda } C _{\lambda ,q} f . 
\end{equation}
By the Weil estimates for Kloosterman sums \eqref{e:K<}, and Plancherel, we have 
\begin{align*}
\langle M _{2,2} \rangle _{E,2}& \lesssim _{\epsilon }  \langle f \rangle _{E} ^{1/2} \sum_{N\leq q \leq \lambda } q ^{ \frac{1-d}2 + \epsilon } 
\rho (q, \lambda ) 
\lesssim  _{\epsilon }
\langle f \rangle _E ^{1/2} 
N ^{\epsilon + \frac{3-d}2} \sigma _{-1/2} (\lambda ^2 ) . 
\end{align*}
The last estimate uses \eqref{e:rho}.

Turn to the `small $ q$' term.   This requires additional contributions to the $ M_1$ and $ M_2$ terms.  
Write 
$ c _{\lambda , q} = c _{\lambda , q} ^{1} + c _{\lambda , q} ^2 $, 
where 
\begin{equation}\label{e:xcq1}
c _{\lambda , q} ^1 (\xi ) = 
\sum_{\ell \in \mathbb Z_q^d} 
K (\lambda ,  q,  \ell )  \Phi _{ \lambda q/N} (\xi - \ell /q) \widetilde  {d \sigma _{\lambda }} ( \xi - \ell /q). 
\end{equation}
We have inserted an additional cutoff term $\Phi _{ \lambda q/N} $ above. 
Then, our third contribution to $ M _{2}$ is the high frequency term $ M _{2,3} = \sum_{q \leq N }  C _{\lambda ,q} ^{2}f$. 
Using the stationary decay estimate \eqref{e:stationary} and the Kloosterman refinement \eqref{e:K<} to see that 
\begin{align*}
\langle M _{2,3} \rangle _{E,2} & \lesssim _{\epsilon } \langle f \rangle_E ^{1/2} \sum_{q\leq N} 
(q/N) ^{ \frac{d-1}2} q ^{\epsilon +\frac{1-d}2} \rho (\lambda ^2 , q)  
\\
& \lesssim  _{\epsilon }  \langle f \rangle_E ^{1/2}  N ^{\frac{1-d}2} \sum_{q\leq N}  q ^{\epsilon } \rho (\lambda ^2 ,q) 
\lesssim  _{\epsilon }\langle f \rangle_E ^{1/2}   
N ^{\epsilon + \frac{3-d}2} \sigma _{- 1/2 } (\lambda ^2 ) . 
\end{align*}
The last estimate follows from \eqref{e:rho2}. 

Then the main point is the last contributions to $ M _{1}$ below.  The definition of $ M _{1,2}$ is of the form to which 
\eqref{p:mq} applies.  
\begin{align}
M _{1,2}(n) &  \leq  \sum_{q \leq N} q  \cdot 
 \widecheck \Phi _{ \lambda q/N} \ast d \sigma _{\lambda } \ast f (n)
\\ &\lesssim  N \langle f \rangle _{E} \sum_{q\leq N} 1 \lesssim N ^2 \langle f \rangle_E. 
\end{align}
 Observe that $\Phi _{ \lambda q/N} \ast d \sigma _{\lambda } \ast f $ is an average of $ f$ over an annulus of radius $ \lambda $, and width $ \lambda q/N$. 
This is compared to $ \langle f \rangle_E$, with loss of $ N/q$. Our proof of \eqref{e:xM1} and \eqref{e:xM2} is complete. 

\section{Complements to the Main Theorems} \label{s:complements}

Concerning sharpness of the  $ \ell ^{p}$ improving estimates in Theorem~\ref{t:improve}, 
the best counterexample we have been able to find shows that  if one has the inequality below, 
\begin{equation} \label{e:assume}
\lVert A _{\lambda } f\rVert _{p'} \lesssim \lambda ^{ d (1- \frac{2}p)} \lVert f\rVert _{p}, 
\end{equation}
valid for all $ \lambda $, then necessarily $ p \geq \frac{d+2}{d}$.  provided $ d \geq 5$.  

Indeed, take $ \lambda ^2 $ to be odd, and  let  $ f $ be the indicator of the sphere of radius $ \lambda $.  
Use the fact that $ A _{\lambda } f (0) \simeq 1$.  

But, in the case of $ d\geq 5$, also take  $ g$ to be the indicator of the set $ G_ \lambda  = \{ A _{\lambda } f > c/ \lambda \}$, 
for appropriate choice of constant $ c$.  That is, $ G_ \lambda $ is the set of $ x$'s for which $ \mathbb S _{\lambda } \cap  x+ \mathbb S _{\lambda }$ 
has about the expected cardinality of $ \lambda ^{d-3}$.

We claim that $ \lvert  G_ \lambda \rvert \gtrsim \lambda   $.  
For an choice of $ 0<  x_1 < \lambda /2$ divisible by $ 4$,
note that there are about $ \lambda ^{d-3}$ points $ (x_2 ,\dotsc, x _{d}) \in \mathbb Z ^{d-1}$ 
of magnitude $ \sqrt {\lambda ^2 - (x_1/2) ^2 }$.  From this, we see that 
\begin{equation*}
\lVert  (x_1 , 0,\dotsc, 0) - (x_1/2, y_2 ,\dotsc, y_d)\rVert = \lambda . 
\end{equation*}
That is,  $ (x_1, 0 ,\dotsc, 0) \in G_ \lambda $.

 We also have an upper bound for $ G$. Apply the $ \ell ^{p}$ improving inequality \eqref{e:improve} 
to $ f = \mathbf 1_{S _{\lambda }}$ to see that 
for $ 0<  \epsilon < 1$, 
\begin{equation} \label{e:generic}
\lvert  G \rvert =  \lvert  \{ A _{\lambda } \mathbf 1_{\mathbb S   _{\lambda }} > c / \lambda \}\rvert 
\lesssim  \lambda ^{ \frac{d+3} {2} + \epsilon }, \qquad \lambda^2 \in \mathbb{N}. 
\end{equation}
Is this estimate sharp?  Notice that this estimate concerns the set of solutions $ n$ to a \emph{pair of}  quadratic equations 
below in which  $ x = (x_1 ,\dotsc, x_d)$ is fixed. 
\begin{align*}
n_1 ^2 + \cdots + n _d ^2 &= \lambda ^2 ,
\\
(n_1 -x_1) ^2 + \cdots + (n _d -x_d) ^2 &= \lambda ^2 , 
\end{align*}
Moreover, we require of $ x$ that the set of possible solutions $ n$ should be of about the expected cardinality. 
We could not find this estimate in the literature.

\bibliographystyle{alpha,amsplain}	


\end{document}